\documentclass[11pt,a4paper,english]{article}
\usepackage{amsmath}
\usepackage{bbm}
\usepackage{latexsym}
\usepackage[english]{babel}
\usepackage{amsfonts}
\usepackage[T1]{fontenc}
\usepackage[latin1]{inputenc}
\usepackage{amsthm}
\usepackage{amssymb}
\usepackage{dsfont}
\usepackage{version}
\usepackage{fancyhdr}
\usepackage{caption}
\usepackage[psamsfonts]{eucal}
\usepackage[numbers]{natbib}
\usepackage{stmaryrd}
\usepackage{url}
\usepackage{hyperref}
\usepackage{color}
\numberwithin{equation}{section}

\newcommand{\R}{\mathbb{R}}

\newtheorem{theorem}{Theorem}[section]
\newtheorem{lemma}[theorem]{Lemma}
\newtheorem{Def}[theorem]{Definition}
\newtheorem{prop}[theorem]{Proposition}

\newtheorem{rem}[theorem]{Remark}

\DeclareMathOperator*{\esup}{ess\,sup}

\def \H {{\mathcal H}}
\def \G {{\mathcal G}}
\def \A {{\mathcal A}}

    \def\cliprn{C_{b,Lip}(\mathbb{R}^n)}
    \def\cliprdn{C_{b,Lip}(\mathbb{R}^{d\times n})}

        \def\lipt{Lip(\Omega_t)}
        
    \def\lip{Lip(\Omega)}
        
    \def\LGp{L^p_G(\Omega)}

            \def\atheta{\mathcal{A}^{\Theta}_{0,\infty}}

\def \E {{\mathbb E}}
\def \hE {\hat{\mathbb E}}
\def \F {{\mathcal F}}
\def \P {{\mathbb P}}
\def \S {\mathcal{S}}
\def \Q {{\mathbb Q}}
\def \R {{\mathbb R}}
\def \N {{\mathbb N}}
    \def\I{\mathds{1}}

\newcommand{\mail}[1]{\href{mailto:#1}{\texttt{#1}}}
\title{Optimal control with delayed information flow of systems driven by $G$-Brownian motion}
\author{Francesca Biagini\footnote{Department of Mathematics, University of Munich, Theresienstra{\ss}e 39, 80333 Munich, Germany.
Email: \mail{francesca.biagini@math.lmu.de}.}\and Thilo Meyer-Brandis\footnote{Department of Mathematics, University of Munich, Theresienstra{\ss}e 39, 80333 Munich, Germany.
Email: \mail{meyerbr@math.lmu.de}.}\and Bernt {\O}ksendal\footnote{Department of Mathematics, University of Oslo, Box 1053 Blindern, N-0316 Oslo, Norway.
Email: \mail{oksendal@math.uio.no}. The research leading to these results has received funding from the
European Research Council under the European Community's Seventh
Framework Program (FP7/2007-2013) / ERC grant agreement no [228087].}
\and Krzysztof Paczka\footnote{Department of Mathematics, University of Oslo, Box 1053 Blindern, N-0316 Oslo, Norway.
Email: \mail{k.j.paczka@cma.uio.no}.}}
\date{8 February 2014}
\begin{document}
\maketitle

\begin{abstract}
In this paper we study  strongly robust optimal control problems under volatility uncertainty. In the $G$-framework we  adapt the stochastic maximum principle  to find  necessary and sufficient conditions for the existence of  a strongly robust optimal control.   \newline

\end{abstract}

\section{Introduction}\label{section:introduction}
One of the motivations for this paper is to study the problem of optimal consumption and optimal portfolio allocation  in finance under model uncertainty. In particular we  focus here on  volatility uncertainty, i.e. a situation where the volatility affecting the asset price dynamics is unknown and we need to consider a family of different volatility processes instead of just one fixed process (and hence also a family of models related to them). 

Volatility uncertainty has been investigated in the literature by following two approaches, i.e. by introducing an abstract sublinear expectation space with a special process called $G$-Brownian motion (see \cite{Peng_GBM}, \cite{Peng_skrypt}), or by quasi-sure analysis (see \cite{Martini}).   In \cite{Denis_function_spaces} it is proven that these two methods are strongly related. The link between these two approaches is the representation of the sublinear expectation $\hE$ associated with the $G$-Brownian motion as a supremum of ordinary expectations over a tight family of probability measures $\mathcal P$, whose elements are mutually singular:
\[
	\hE[.]=\sup_{\P\in\mathcal P}\E^{\P}[.],
\]
see \eqref{set_measure} and Theorem \ref{tw_denis} for more details.

In this paper we work in  a $G$-Brownian motion setting as in \cite{Peng_GBM} and  use the related stochastic calculus, including the It\^o formula, $G$-SDE's, martingale representation and $G$-BSDE's,  as developed in \cite{Peng_GBM}, \cite{Peng_skrypt}, \cite{Soner_mart_rep}, \cite{Song_Mart_decomp}, \cite{Soner_quasi_anal}, \cite{Complete_mart}, \cite{Peng_bsde}, \cite{Girsanov}. It is important for understanding the nature of the $G$-Brownian motion to note that its quadratic variation $\langle B\rangle$ is not deterministic,  but it is absolutely continuous with the density taking value in a fixed set (for example $[\underline{\sigma}^2,\bar \sigma^2]$ for $d=1$). Each $\P\in\mathcal P$ can be seen then as a model with a different scenario for the quadratic variation. That justifies why $G$-Brownian motion is a good framework for investigating model uncertainty.

In  a $G$-Brownian motion setting one considers the following stochastic optimal control problem: to find the control $\hat u\in\mathcal A$ such that
\begin{equation}\label{opt_problem}
	J(\hat u)=\sup_{u\in\mathcal A}\,  J(u),
\end{equation}
with
\begin{align}
	J(u):&=\hE[\int_0^T f(t,X^u(t),u(t))dt+g(X^u(T))] \label{functional}\\
	&=\sup_{\P\in\mathcal{P}}\, \E^{\P}[\int_0^T f(t,X^u(t),u(t))dt+g(X^u(T))] =:  \sup_{\P\in\mathcal{P}} J^{\P}(u), \nonumber
\end{align}
where $X^u$ is a controlled $G$-SDE, see \eqref{eq_process}. 
This problem has been studied in \cite{utility1}, \cite{utility2}.  In  \cite{utility2} they show that the value function associated with such an optimal control problem satisfies the dynamic programming principle and is a viscosity solution of some HJB equation.\footnote{To be exact, the authors considered a more general problem of recursive utility.} \cite{utility1} investigates the robust investment problem for geometric $G$-Brownian motion and  2BSDE's (which is a version of $G$-BSDE's) are used to find an optimal solution. In both papers the optimal control is robust in the worst case scenario sense. 

It is interesting to note that in the simplest example of the optimal portfolio problem, which is the Merton problem with the logarithmic utility, one can easily prove that there exists a portfolio which is optimal not only in the worst case scenario, but also for all probability measures $\P$ (with the optimality criterion $J^{\P}$). We call this \emph{a strongly robust control}. This strongly robust control is thus optimal in a much more robust sense than the worst case scenario optimality. 
The new strongly robust optimality uses the fact that probability measures $\P$ are mutually singular, hence one can modify the $\P$-optimal control $\hat u^{\P}$ outside the support of a probability measure $\P$ without losing the $\P$-optimality.  As a consequence, if the family $\{\hat u^{\P}\}_{\P \in \mathcal{P}}$ satisfies some consistency conditions, the controls can be aggregated into a unique control $\hat u$, which is optimal under every probability measure $\P$. See \cite{Soner_quasi_anal} for more details on aggregation. 

In this paper we study  strongly robust optimal control problems. However, instead of checking the consistency condition for the family of controls  and using the aggregation theory established in \cite{Soner_quasi_anal}, we  adapt the stochastic maximum principle to the $G$-framework to find  necessary and sufficient conditions for the existence of  a strongly robust optimal control.

The paper is structured in the following way. In Section 2 we give a quick overview on the $G$-framework. Section 3 is devoted to a sufficient maximum principle in the partial information case. In Section 4 we investigate the necessary maximum principle for the full-information case. In Section 5 we give three examples, including the Merton problem with the logarithmic utility, already mentioned earlier. In Section 6 we provide a counter-example and  show that it is not possible to relax the crucial assumption of the sufficient maximum principle without losing the strongly robust sense of optimality.

\section{Preliminaries}\label{section:preliminaries}
	Let $\Omega$ be a given set and $\H$  be a vector lattice of real functions defined on $\Omega$, ie. a linear space containing $1$ such that $X\in\H$ implies $|X|\in\H$. We will treat elements of $\H$ as random variables. 
	\begin{Def}\label{def_sublinear_exp}
		\emph{A sublinear expectation} $\E$ is a functional $\E\colon \H\to\R$ satisfying the following properties
		\begin{enumerate}
			\item \textbf{Monotonicity:} If $X,Y\in\H$ and $X\geq Y$ then $\E [X]\geq\E [Y]$.
			\item \textbf{Constant preserving:} For all $c\in\R$ we have $\E [c]=c$.
			\item \textbf{Sub-additivity:} For all $X,Y\in\H$ we have $\E [X] - \E[Y]\leq\E [X-Y]$.
			\item \textbf{Positive homogeneity:} For all $X\in\H$  we have $\E [\lambda X]=\lambda\E [X]$, $\forall\,\lambda\geq0$.
		\end{enumerate}
		The triple $(\Omega,\H,\E)$ is called \emph{a sublinear expectation space}. 
	\end{Def}
	
We will consider a space $\H$ of random variables having the following property: if $X_i\in\H,\ i=1,\ldots n$ then
	\[
		\phi(X_1,\ldots,X_n)\in\H,\quad \forall \phi\in\cliprn,
	\]
	where $\cliprn$ is the space of all bounded Liptschitz continuous functions on $\R^n$. 
	\begin{Def}
		An $m$-dimensional random vector $Y=(Y_1,\ldots,Y_m)$ is said to be independent of an $n$-dimensional random vector $X=(X_1,\ldots,X_n)$ if for every $\phi\in C_{b,Lip}(R^n\times R^m)$ 
		\[
			\E[\phi(X,Y)]=\E[\E[\phi(x,Y)]_{x=X}].
		\]
		Let $X_1$ and $X_2$ be  $n$-dimensional random vectors defined on sublinear random spaces $(\Omega_1,\H_1,\E_1)$ and $(\Omega_2,\H_2,\E_2)$ respectively. We say that $X_1$ and $X_2$ are identically distributed and denote it by $X_1 \sim X_2$, if for each $\phi\in\cliprn$ one has
		\[
			\E_1[\phi(X_1)]=\E_2[\phi(X_2)].
		\]
	\end{Def}
	\begin{Def}
		A $d$-dimensional random vector $X=(X_1,\ldots,X_d)$ on a sublinear expectation space $(\Omega,\H,\E)$ is said to be $G$-normally distributed if for each $a,b\geq0$ and each $Y\in\H$ such that $X\sim Y$ and $Y$ is independent of $X$, one has
		\[
			aX+bY \sim \sqrt{a^2+b^2}X.
		\]
		The letter $G$ denotes a function defined as
		\[
			G(A):=\frac{1}{2}\E[(AX,X)]\colon \S_d\to \R,
		\]
		where $\S_d$ is the space of all $d\times d$ symmetric matrices. We assume that $G$ is non-degenerate, i.e. $G(A)-G(B)\geq \beta\operatorname{tr}[A-B]$ for some $\beta>0$.
	\end{Def}
	It can be checked that $G$ might be represented as
	\begin{equation}\label{eq_g_exp_representation_theta}
		G(A)=\frac{1}{2}\sup_{\gamma \in\Theta} \textrm{tr}\, (\gamma\gamma^{T}A),	
	\end{equation}
	where $\Theta$ is a non-empty bounded and closed subset of $\R^{d\times d}$.
	\begin{Def}
		Let $G\colon \S_d\to\R$ be a given monotonic and sublinear function. A stochastic process $B=(B_t)_{t\geq 0}$ on a sublinear expectation space $(\Omega,\H,\E)$ is called a $G$-Brownian motion if it satisfies following conditions
		\begin{enumerate}
			\item $B_0=0$,
			\item $B_t\in\H$ for each $t\geq0$.
			\item For each $t,s\geq0$ the increment $B_{t+s}-B_t$ is $G$-normally distributed and independent of $(B_{t_1},\ldots, B_{t_n})$ for each $n\in\N$ and $0\leq t_1<\ldots<t_n\leq t$.
		\end{enumerate}		 
	\end{Def}
	\begin{Def}
		Let $\Omega=C_0(\R_{+},\R^d)$, i.e. the space of all $\R^d$-valued continuous functions starting at $0$. We equip this space with the uniform convergence on compact intervals topology and denote by $\mathcal{B}(\Omega)$ the Borel $\sigma$-algebra of $\Omega$. Let
		\[
			\H=\lip:=\{\phi(\omega_{t_1},\ldots,\omega_{t_n})\colon \forall n\in\N,  t_1,\ldots,t_n\in[0,\infty)\ \textrm{and}\ \phi\in\cliprdn\}.
		\]
		A $G$-expectation $\hE$ is a sublinear expectation on $(\Omega,\H)$ defined as follows: for $X\in\lip$ of the form
		\[
			X=\phi(\omega_{t_1}-\omega_{t_0},\ldots,\omega_{t_n}-\omega_{t_{n-1}}), \quad 0\leq t_0<t_1<\ldots<t_n,
		\]
		we set
		\[
			\hE[X]:=\E[\phi(\xi_1\sqrt{t_1-t_0},\ldots,\xi_n\sqrt{t_n-t_{n-1}})],
		\]
		where $\xi_1,\ldots\xi_n$ are $d$-dimensional random variables on sublinear expectation space $(\tilde{\Omega},\tilde{\H},\E)$ such that for each $i=1,\ldots,n$ $\xi_i$, is $G$-normally distributed and independent of $(\xi_1,\ldots,\xi_{i-1})$. We denote by $\LGp$ the completion of $\lip$ under the norm $\|X\|_p:=\hE[|X|^p]^{1/p}$, $p\geq1$. Then it is easy to check that $\hE$ is also a sublinear expectation on the space $(\Omega,\LGp)$, $\LGp$ is a Banach space and the canonical process $B_t(\omega):=\omega_t$ is a $G$-Brownian motion.
	\end{Def}
	Following \cite{Peng_skrypt} and \cite{Denis_function_spaces}, we introduce the notation: for each $t\in[0,\infty)$
	\begin{enumerate}
		\item $\Omega_t:=\{w_{.\wedge t}\colon \omega\in\Omega\}$, $\F_t:=\mathcal{B}(\Omega_t)$, 
		\item $L^0(\Omega)\colon$ the space of all $\mathcal{B}(\Omega)$-measurable real functions,
		\item $L^0(\Omega_t)\colon$ the space of all $\mathcal{B}(\Omega_t)$-measurable real functions,
		\item $\lipt:=\lip \cap L^0(\Omega_t)$, $L^p_G(\Omega_t):=L^p_G(\Omega)\cap L^0(\Omega_t)$,
		\item $M^2_G(0,T)$ is the completion of the set of elementary processes of the form
		\[
			\eta(t)=\sum_{i=1}^{n-1}\xi_i\I_{[t_i,t_{i+1})}(s),
		\]
		where $0\leq t_1<t_2<\ldots<t_n\leq T,\ n\geq 1$ and $\xi_i\in Lip(\Omega_{t_i})$. The completion is taken under the norm
		\[
			\|\eta\|^2_{M^2_G(0,T)}:=\hE[\int_0^T|\eta(t)|^2ds].
		\]
	\end{enumerate}
			\begin{Def}\label{def_cond_exp}
		Let $X\in\lip$ have the representation
		\[
			X=\phi(B_{t_1},B_{t_2}-B_{t_1},\ldots,B_{t_n}-B_{t_{n-1}}),\quad \phi\in C_{b,Lip}(\R^{d\times n}),\ 0\leq t_1<\ldots<t_n<\infty.
		\]
		We define the conditional $G$-expectation under $\mathcal{F}_{t_j}$ as
		\[
			\hE[X|\F_{t_j}]:=\psi(B_{t_1},B_{t_2}-B_{t_1},\ldots,B_{t_j}-B_{t_{j-1}}),
		\]
		where
		\[
			\psi(x):=\hE[\phi(x,B_{t_{j+1}}-B_{t_{j}},\ldots,B_{t_n}-B_{t_{n-1}})].
		\]
	\end{Def}
	Similarly to the $G$-expectation,  the conditional $G$-expectation might be also extended to the sublinear operator $\hE[.|\F_t]\colon L^p_G(\Omega)\to L^p_G(\Omega_t)$ using the continuity argument. For more properties of the conditional G-expectation, see \cite{Peng_skrypt}.

	$G$-(conditional) expectation plays a crucial role in the stochastic calculus for $G$-Brownian motion. In \cite{Denis_function_spaces} it was shown that the analysis of the $G$-expectation might be embedded in the theory of upper-expectations and capacities. 
	\begin{theorem}[\cite{Denis_function_spaces}, Theorem 52 and 54]\label{tw_denis}
		Let $(\tilde{\Omega}, \G, \P_0)$ be a probability space carrying a standard $d$-dimensional Brownian motion $W$ with respect to its natural filtration $\mathbb{G}$. Let $\Theta$ be a representation set defined as in eq. (\ref{eq_g_exp_representation_theta}) and denote by $\atheta$ the set of all $\Theta$-valued $\mathbb{G}$-adapted processes on an interval $[0,\infty)$. For each $\theta\in\atheta$ define $\P^{\theta}$ as the law of a stochastic integral $\int_0^.\,\theta_sdW_s$ on the canonical space $\Omega=C_0(\R_{+},\R^d)$. We introduce the sets
		\begin{equation}\label{set_measure}
			\mathcal{P}_1:=\{\P^{\theta}\colon \theta \in\atheta\}, \quad\textrm{and}\quad\mathcal{P}:=\overline{\mathcal{P}_1},
		\end{equation}
where the closure is taken in the weak topology. $\mathcal{P}_1$ is tight, so $\mathcal{P}$ is weakly compact. Moreover, one has the representation
		\begin{equation}\label{eq_rep_g_exp_upper}
			\hE[X]=\sup_{\P\in\mathcal{P}_1}\, \E^{\P}[X]=\sup_{\P\in\mathcal{P}}\, \E^{\P}[X],\quad \textrm{for each }X\in L_G^1(\Omega).
		\end{equation}
	\end{theorem}
For convenience we will always consider only a Brownian motion on the canonical space $\Omega$ with the Wiener measure $\P_0$. \newline
Similarly an analogous  representation holds for the $G$-conditional expectation.
	\begin{prop}[\cite{Soner_mart_rep}, Proposition 3.4]\label{prop_rep_conditional_G_exp}
		Let $\mathcal{P}(t,P):=\{\P'\in\mathcal{P}\colon \P=\P'\ \textrm{on}\ \F_t\}$. Then for any $X\in L^1_G(\Omega)$ one has 
		\begin{equation}\label{eq_cond_exp_rep}
			\hE[X|\F_t]=\sideset{}{^{\P}}\esup_{\P'\in\mathcal{P}(t,\P)}\, \E^{\P'}[X|\F_t],\ \P-a.s.
		\end{equation}
		\end{prop}
\noindent We now introduce the Choquet capacity (see \cite{Denis_function_spaces}) related to $\mathcal{P}$
\[
	c(A):=\sup_{\P\in\mathcal{P}}\, \P(A),\quad A\in \mathcal{B}(\Omega).
\]
\begin{Def}
\begin{enumerate}
\item	A set $A$ is said to be polar, if $c(A)=0$. Let $\mathcal{N}$ be a collection of all polar sets. A property is said to hold quasi-surely (abbreviated to q.s.) if it holds outside a polar set. 
\item	We say that a random variable $Y$ is a version of $X$ if $X=Y$ q.s.	
\item	A random variable $X$ is said to be quasi-continuous (q.c. in short), if for every $\varepsilon>0$ there exists an open set $O$ such that $c(O)<\varepsilon$ and $X|_{O^c}$ is continuous.
\end{enumerate}
\end{Def}
\noindent We have the following characterization of spaces $L^p_G(\Omega)$. This characterization shows that $L^p_G(\Omega)$ is a rather small space.
\begin{theorem}[Theorem 18 and 25 and in \cite{Denis_function_spaces}]
	For each $p\geq1$ one has
	\[
		L^p_G(\Omega)=\{X\in L^0(\Omega)\colon X\textrm{ has a q.c. version and }\lim_{n\to\infty}\, \hE[|X|^p\I_{\{|X|>n\}}]=0\}.
	\]
\end{theorem}
\noindent  The $G$-expectation turns out to be a good framework to develop stochastic calculus of the It\^o type.
We can have also $G$-SDE's and a version of the backward SDE's. As backward equations are a key tool to consider the maximum principle, we now give some short introduction to $G$-BSDE's and their properties (for simplicity in a one-dimensional case). 

Fix two functions $f,g:\Omega\times [0,T]\times\R\times\R\to \R$ and $\xi\in L^p_G(\Omega_T),\ p>2$. We will say that the triple $(p^G,q^G,K)$ is a solution of the $G$-BSDE with drivers $f,g$ and terminal condition $\xi$ if
\begin{align}
		dp^G(t)&=-f(t,p^G(t),q^G(t))dt-g(t,p^G(t),q^G(t))d\langle B\rangle (t)+q^GdB(t)+dK(t), \label{eq_adj}\\
	p^G(T)&=\xi,\nonumber
\end{align}
where $K$ is a non-increasing $G$-martingale starting at $0$. In \cite{Peng_bsde} 
the existence and uniqueness of such a $G$-BSDE are proved under some Lipschitz and regularity  conditions on the driver.

Furthermore under any $\P\in\mathcal{P}$ the process $p^G$  is a supersolution of a classical BSDE with drivers $f$ and $g$ and terminal condition $\xi$ on the probability space $(\Omega,\F,\P)$ (we will call such a BSDE a $\P$-BSDE). Hence, by comparison theorem for supersolutions and solutions we get
\[
	p^G(t)\geq p^{\P}(t)\quad \P-a.s.,
\]
where $p^{\P}$ is a solution of $\P$-BSDE. It might be also checked that $p^G$ is minimal in the  sense that
\[
	p^G(t)=\sideset{}{^{\P}}\esup_{\Q\in\mathcal{P}(t,\P)}\, p^{\Q}(t)\quad \P-a.s. ,
\]
see \cite{Soner_2bsde} for this representation. From now on we drop the superscript $G$ in the notation for $G$-BSDE's whenever this doesn't lead to  confusion.

\section{A sufficient maximum principle}\label{section:suff_maxprinciple}

Let $B(t)$ be a $G$-Brownian motion with associated sublinear expectation operator $\hE$. We consider controls $u$ taking values in a closed convex set $U\subset\R$. Let $X(t)=X^u(t)$ be a controlled process of the form

\begin{align}\label{eq_process}
dX(t)&= b(t,X(t),u(t)) dt + \mu(t,X(t),u(t)) d\langle B\rangle_t + \sigma(t,X(t),u(t)) dB(t); \ 0\leq t\leq T,\\
X(0)&=x\in\R.\nonumber
\end{align}
We assume that the coefficients $b,\ \mu,\ \sigma$ are Lipschitz continuous w.r.t. the space variable uniformly in $(t,u)$. Moreover, if the coefficients are not deterministic, they must belong to the space $M^2_G(0,T)$ for each $(x,u)\in \R\times U$ .

Let $f:[0,T]\times\R\times U\to \R$ and $g:\R\to\R$ be two measurable functions such that $f$ is continuous w.r.t the second variable and $g$ is a lower-bounded, differentiable function with quadratic growth s.t. there exists a constant $C>0$ and $\epsilon>0$ s.t
\[
	|g'(x)|<C(1+|x|)^{\frac{1}{1+\epsilon/2}}.
\] 
We let $\A$ denote the set of all admissible controls. For $u$ to be in  $\A$ we require that $u$ is quasi-continuous and adapted to $(\mathcal{F}_{(t-\delta)^+})_{t\geq \delta}$, where $\delta \geq 0$ is a given constant. This means that our control $u$ has only access to a delayed information flow. Moreover, we assume that for each $u\in\A$ the following integrability condition is satisfied
\[
	\hE\left[\int_0^Tf(t,X(t),u(t))dt\right]<\infty.
\]
Then for each $\P\in\mathcal{P}$, the performance functional associated to $u\in\A$ is assumed to be of the form
\begin{equation}
J^{\P}(u)=\E^{\P}\left[\int_0^T f(t,X(t),u(t)) dt +g(X(T))\right].
\end{equation}
We study the following strongly robust optimal control problem: find $\hat u\in \A$ such that

\begin{equation}\label{eq_sup}
\sup_{u\in \A}J^{\P}(u)= J^{\P}(\hat u)\quad \forall\ \P\in\mathcal{P},
\end{equation}
where the set $\mathcal{P}$ is introduced in \eqref{set_measure}.
To this end we define the Hamiltonian

\begin{equation}
H(t,x,u,p,q)= f(t,x,u) + \left[b(t,x,u) + \mu(t,x,u)\frac{d\langle B\rangle_t }{dt}\right] p
+\sigma(t,x,u) \frac{d\langle B\rangle_t }{dt} q,
\end{equation}
and  the associated $G$-BSDE with adjoint processes $p(t),q(t),K(t)$ by

\begin{align}\label{eq_p}
dp(t)&= -\frac{\partial H }{\partial x}(t) dt + q(t) dB(t) + dK(t);\ 0\leq t\leq T,\\
p(T)&=g'(X(T)).\nonumber
\end{align}
Note that the solution of such $G$-BSDE exists thanks to the assumption on the functions $f$ and $g$ and on the definition of the admissible control (see \cite{Peng_bsde} for details).

\begin{theorem}\label{theo_suff}
Let $\hat u\in \A$ with corresponding solution $\hat X(t), \hat p(t),\hat q(t), \hat K(t)$ of \eqref{eq_process} and \eqref{eq_p}
in \eqref{eq_p} such that $\hat K\equiv 0$. Assume that:

\begin{equation}
(x,u)\rightarrow H(t,x,u,\hat p(t),\hat q(t)) \ \textrm{and}\ x \rightarrow g(x) \
\textrm{are concave for all $t$ a.s.},
\end{equation}
\noindent and
\begin{equation}\label{eq_assumption_suff_max_princip}
\hE\left[\pm  \frac{\partial}{\partial u} H(t,\hat X(t),u,\hat p(t),\hat q(t))|_{u=\hat u(t)}    |\F_{(t-\delta)^{+}}\right]=0.
\end{equation}
for all $t$ q.s.
Then $\hat u=u$ is a strongly robust optimal control for the problem \eqref{eq_sup}.

\end{theorem}

\begin{proof}
For the sake of simplicity , in the sequel we  adopt the concise notation $f(t):=f(t,X^u(t),u(t))$, $\hat f(t)=f(t,X^{\hat u}(t),\hat u(t))$, $X(T)=X^u(T)$, $\hat X(T)=X^{\hat u}(T)$.
Let $u\in\A$ be arbitrary and consider

\begin{align}
\sup_{\P\in\mathcal P} \{J^{\P}(u)-J^{\P}(\hat u)\}&= \sup_{\P\in\mathcal P}\, \E^{\P}\left[\int_0^T( f(t) -\hat f(t))dt +g(X(T))-g(\hat X(T))\right] \nonumber \\ 
&=\hE\left[\int_0^T( f(t) -\hat f(t))dt +g(X(T))-g(\hat X(T))\right] \nonumber \\ 
& = \hE[ I_1 +I_2], \label{eq_12}
\end{align}
where  $J$ is introduced in \eqref{functional} and

\begin{equation*}
I_1:=\int_0^T( f(t) -\hat f(t))dt,\quad I_2:=g(X(T))-g(\hat X(T)).
\end{equation*}
By definition of $H$ we can write
\small{
\begin{equation}
I_1=\int_0^T \left\{H(t) -\hat H(t)-\left[b(t) -\hat b(t) + (\mu(t) -\hat \mu(t))\frac{d\langle B\rangle_t }{dt}\right]\hat p(t)
 - [\sigma(t) -\hat \sigma(t)]\frac{d\langle B\rangle_t }{dt} \hat q(t)\right\} dt. \label{eq_sigma1}
\end{equation}
}
By concavity of $g$, \eqref{eq_p}  and the It\^o formula we have

\begin{align}
I_2&\leq g'(\hat X(T))(X(T)-\hat X(T))=\hat p(T)(X(T)-\hat X(T)) \nonumber \\ 
&= \int_0^T\hat p(t)d(X(t)-\hat X(t)) + \int_0^T(X(t)-\hat X(t)) d\hat p(t) + \int_0^T d\langle \hat p, X-\hat X\rangle(t)\nonumber  \\ 
&= \int_0^T\hat p(t) [b(t) -\hat b(t) + (\mu(t) -\hat \mu(t))\frac{d\langle B\rangle_t }{dt}] dt\nonumber \\ 
& \quad + \int_0^T(X(t)-\hat X(t)) (-\frac{\partial \hat H }{\partial x}(t)) dt + \int_0^T[\sigma(t) -\hat \sigma(t)]\frac{d\langle B\rangle_t }{dt} \hat q(t) dt\\
& \quad + \int_0^T\hat p(t)[\sigma(t)-\hat \sigma(t)]dB(t)+\int_0^T[X(t)-\hat X(t)]\hat q(t)dB(t) \,.\label{eq_sigma2}
\end{align}
Adding \eqref{eq_sigma1} and  \eqref{eq_sigma2} and using concavity of $H$ we get,
by  the sublinearity of the $G$-expectation and by \eqref{eq_12}, that

\begin{align*}
\sup_{\P\in\mathcal P} \{J^{\P}(u)-J^{\P}(\hat u)\}&\leq \hE\left[\int_0^T\left(\hat p(t)[\sigma(t)-\hat \sigma(t)]+[X(t)-\hat X(t)]\hat q(t)\right)dB(t)\right] \\
&\quad + \hE\left[\int_0^T  [H(t) -\hat H(t) -\frac{\partial \hat H }{\partial x}(t) (X(t)-\hat X(t))] dt \right]\\
&\leq \hE\left[\int_0^T \frac{\partial \hat H }{\partial u}(t)(u(t)-\hat u(t)) dt \right] \\
&\leq \int_0^T\hE\left[ \frac{\partial \hat H }{\partial u}(t)(u(t)-\hat u(t)) \right] dt \\
&\leq \int_0^T\hE\left[ \hE\left[\frac{\partial \hat H }{\partial u}(t)(u(t)-\hat u(t)) | \mathcal{F}_{(t-\delta)^{+}}\right]\right] dt\\
&\leq \int_0^T\hE\left[ \hE\left[\frac{\partial \hat H }{\partial u}(t)| \mathcal{F}_{(t-\delta)^{+}}\right](u(t)-\hat u(t))^{+}\right.\\
&+\left.\hE\left[-\frac{\partial \hat H }{\partial u}(t)| \mathcal{F}_{(t-\delta)^{+}}\right](u(t)-\hat u(t))^{-}\right] dt= 0,
\end{align*}
since $u=\hat u$ is a critical point of the Hamiltonian. This proves that $\hat u:=\hat u$ is optimal.
\end{proof}

\begin{rem}
	Note that if $\delta=0$ we can relax slightly the assumption in eq. \eqref{eq_assumption_suff_max_princip} by just requiring that
	\[
		\max_{v\in U}H(t,\hat X(t),v,\hat p(t),\hat q(t))]=  H(t,\hat X(t),\hat u(t),\hat p(t),\hat q(t)).
	\]
\end{rem}

\section{A necessary maximum principle for full-information case}\label{section:nec_maxprinciple}

It is a drawback of the previous result that the concavity conditions are not satisfied in many applications. Therefore it is of interest to have a maximum principle, which does not need this condition. Moreover, the requirement that the non-increasing $G$-martingale $\hat{K}$ disappears from the adjoint equation for the optimal control $\hat{u}$ is a very strong assumption, which is however crucial in the proof. In this section we prove a result which doesn't depend on the concavity of the Hamiltonian. Moreover, in the Merton problem we show that the necessary maximum principle might be obtained without the assumption on the process $\hat K$. We make the following assumptions:

\begin{enumerate}
\item[A1.] for all $u,\beta\in\A$ with $\beta$ bounded, there exists $\delta >0$ such that 
\begin{equation*}
u+a\beta\in \A , \quad \textrm{for all} \ a\in (-\delta, \delta).
\end{equation*}

\item[A2.] For all $t,h$ such that $0\leq t < t+h\leq T$ and all bounded random variables $\alpha\in L^1_G(\Omega_t)$\footnote{It is easy to see that for a fixed $\P\in\mathcal{P}$ the set of all bounded random variables from the space $L^1_G(\Omega)$ is dense in the space $L^p_{\P}(\Omega_t)$ under the norm $(\E^{\P}[|.|^p])^{1/p}$ for any $p\geq1$.}, the control 
 \begin{equation*}
\beta(s):= \alpha \I_{[t,t+h]} (s)
\end{equation*}
belongs to  $\A$.

\item[A3.] Given $u,\beta\in\A$ with $\beta$ bounded, the derivative process 
 \begin{equation*}
Y(t):=\frac{d}{da} X^{u+\alpha\beta}(t)
\end{equation*}
exists, $Y(0)=0$ and
\begin{align*}
dY(t)&= \left\{ \frac{\partial b}{\partial x}(t) Y(t) + \frac{\partial b}{\partial u}(t) \beta(t)\right\} dt \\
&\quad +  \left\{ \frac{\partial \mu}{\partial x}(t) Y(t) + \frac{\partial \mu}{\partial u}(t) \beta(t)\right\} d\langle B\rangle_t
+ \left\{ \frac{\partial \sigma}{\partial x}(t) Y(t) + \frac{\partial \sigma}{\partial u}(t) \beta(t)\right\} dB(t)\,.
\end{align*}
\end{enumerate}

\begin{lemma}\label{lemma_necessMax}
Assume that A1, A2, A3 hold and that $\hat u$ is an optimal control for the performance functional
\begin{equation*}
u\rightarrow J^{\P}(u)
\end{equation*}
for some probability measure $\P\in\mathcal{P}$. Consider the adjoint equation as a BSDE under probability measure $\P$:
\begin{align}\label{eq_p_under_P}
dp^{\P}(t)&= -\frac{\partial H }{\partial x}(t,X(t),p^{\P}(t), q^{\P}(t)) dt + q^{\P}(t) dB(t);\ 0\leq t\leq T,\\
p^{\P}(T)&=g'(X(T))\quad \P-a.s.\nonumber
\end{align}
Then
 \begin{equation*}
\frac{\partial \hat H^{\P}}{\partial u}(t):=\frac{\partial}{\partial u} H(t,\hat X(t), u, \hat p^{\P}(t),\hat q^{\P}(t))\,|_{u=\hat u(t)}=0.
\end{equation*}
\end{lemma}
\begin{proof}
Consider
\small{
\begin{align}
\frac{d}{d a}J^{\P}(u+a\beta)&=  \frac{d}{d a} \E^{\P}\left[\int_0^T f(t,X^{u+a\beta}(t),u(t)) dt +g(X^{u+a\beta}(T))\right] \nonumber\\
&= \lim_{a\rightarrow 0} \frac{1}{a} \E^{\P}\left[\int_0^T f(t,X^{u+a\beta}(t),u(t)) dt +g(X^{u+a\beta}(T))\right]\nonumber \\
&\quad \quad -\E^{\P}\left[\int_0^T f(t,X(t),u(t)) dt +g(X(T))\right]\nonumber\\
&=  \lim_{a\rightarrow 0}  \E^{\P}\left[\int_0^T \frac{1}{a}\left\{f(t,X^{u+a\beta}(t),u(t)) -f(t,X(t),u(t))\right\}dt + \frac{1}{a}\left\{g(X^{u+a\beta}(T))-g(X(T))\right\}\right]\nonumber\\
&=\E^{\P}\left[\int_0^T \left(\frac{\partial f}{\partial x}(t,X(t),u(t))Y(t) + \frac{\partial f}{\partial u}(t,X(t),u(t))\beta(t)\right)dt +
g'(X(T))Y(T) \right]\,. \label{eq_der}
\end{align}
}
\noindent By the It\^o formula
\small{
\begin{align}
 &\E^{\P}\left[g'(X(T))Y(T)\right] =\E^{\P}\left[p(T)Y(T)\right]\nonumber\\
&= \E^{\P}\left[\int_0^T p^{\P}(t) dY(t) +\int_0^T Y(t) dp^{\P}(t) + \int_0^T q^{\P}(t)\left\{ \frac{\partial \sigma}{\partial x}(t) Y(t) + \frac{\partial \sigma}{\partial u}(t) \beta(t)\right\} d\langle B\rangle_t \right]\nonumber\\
&\leq   \E^{\P}\left[\int_0^T p^{\P}(t)\left\{\frac{\partial b}{\partial x}(t) Y(t) + \frac{\partial b}{\partial u}(t) \beta(t)\right\} dt +  \int_0^T p^{\P}(t) \left\{ \frac{\partial \mu}{\partial x}(t) Y(t) + \frac{\partial \mu}{\partial u}(t) \beta(t)\right\} d\langle B\rangle_t \right. \nonumber\\
& \quad \left. 
 +\int_0^T Y(t)(-\frac{\partial \hat H^{\P}}{\partial x}(t)) dt  + \int_0^T q^{\P}(t)\left\{ Y(t)\frac{\partial \sigma}{\partial x}(t) +\frac{\partial \sigma}{\partial u}(t) \beta(t)\right\} d\langle B\rangle_t \right]\nonumber\\
&=\E^{\P}\left[ \int_0^T Y(t) \left\{ p^{\P}(t)\left(\frac{\partial b}{\partial x}(t) +  \frac{\partial \mu}{\partial x}(t) \frac{d\langle B\rangle_t
}{dt}\right) + q^{\P}(t) \frac{\partial \sigma}{\partial x}(t) - \frac{\partial H^{\P}}{\partial x}(t)\right\}dt \right. \nonumber\\
& \quad \left. 
+ \int_0^T \beta(t)\left\{p^{\P}(t)\left(\frac{\partial b}{\partial u}(t) +  \frac{\partial \mu}{\partial u}(t) \frac{d\langle B\rangle_t
}{dt}\right) + q^{\P}(t) \frac{\partial \sigma}{\partial u}(t) \frac{d\langle B\rangle_t
}{dt} \right\} dt  \right]\,.\label{eq_der2}
\end{align}
}
Adding \eqref{eq_der} and  \eqref{eq_der2} we get
\begin{equation*}
\frac{d}{d a}J^{\P}(u+a\beta)\leq  \E^{\P}\left[ \int_0^T \beta(t)\frac{\partial H^{\P}}{\partial u}(t) dt\right]\,.
\end{equation*}
If $\hat u$ is an optimal control, then the above gives
\begin{equation*}
0=\frac{d}{d a}J^{\P}(\hat u+a\beta)\leq  \E^{\P}\left[ \int_0^T \beta(t)\frac{\partial \hat H^{\P}}{\partial u}(t) dt\right]
\end{equation*}
for all bounded $\beta\in \A$. Applying this to both $\beta$ and $-\beta$, we conclude that
\begin{equation*}
\E^{\P}\left[ \int_0^T \beta(t)\frac{\partial \hat H}{\partial u}(t) dt\right]=0.
\end{equation*}
By A2 together with the footnote about the denseness we can then proceed to deduce that
\begin{equation*}
\frac{\partial \hat H^{\P}}{\partial u}(t)=0\quad \P-a.s.
\end{equation*}
\end{proof}

Using the lemma we can easily get the following necessary maximum principle.
\begin{theorem}
	Assume that A1, A2, A3 hold and that $\hat u$ is a strongly robust optimal control for the performance functional
	\begin{equation*}
	u\rightarrow J^{\P}(u)
	\end{equation*}
	for every probability measure $\P\in\mathcal{P}$. Consider the adjoint equation as a $G$-BSDE:
	\begin{align}\label{eq_p_under_G}
		d\hat p^{G}(t)&= -\frac{\partial H }{\partial x}(t,X(t),\hat p^{G}(t), \hat q^{G}(t)) dt + \hat q^{G}(t) dB(t)+d\hat K(t);\ 0\leq t\leq T,\\
		\hat p^{G}(T)&=g'(X(T))\quad q.s.\nonumber
	\end{align}
If $\hat K\equiv 0\ q.s.$ then
 \begin{equation}\label{eq_necessary_max_hamilton_G}
\frac{\partial \hat H^{G}}{\partial u}(t):=\frac{\partial}{\partial u} H(t,\hat X(t), u, \hat p^{G}(t),\hat q^{G}(t))\,|_{u=\hat u(t)}=0,\ q.s.
\end{equation}
\end{theorem}

\begin{proof}
We now prove that the relation in \eqref{eq_necessary_max_hamilton_G} holds for every $\P\in\mathcal{P}$. 
Fix $\P\in\mathcal{P}$. If $\hat K\equiv 0\ q.s.$ then by the uniqueness of the solution of $\P$-BSDE we get that $\hat p^G\equiv \hat p^{\P}\ \P-a.s.$ and $\hat q^G\equiv \hat q^{\P}\ \P-a.s$. But by Lemma \ref{lemma_necessMax} we know that $\hat u$ is a $\P-a.s.$ critical point of $\hat H^{\P}(t)$ hence also $\hat H^G(t)$. By the arbitrariness of $\P\in\mathcal{P}$ we get the assertion of the theorem.
\end{proof}

Just as we mentioned at the beginning of this section, the assumption on the process $\hat K$ is a big disadvantage. However, if we limit our considerations to the Merton-type problem, we are able to show the necessary maximum principle without this assumption.
\begin{theorem}\label{theorem_necessMax_Merton}
	Assume that
	\begin{enumerate} 
	\item A1, A2, A3 hold. 
	\item $b\equiv 0$, $\mu(t,x,u)=x\cdot u\cdot m(t)$ and $\sigma(t,x,u)=x\cdot u\cdot s(t)$ for some bounded functions $m$ and $s$ such that for each $t\in[0,T]$ $m(t)$ and $s(t)$ are quasi-continuous. Moreover, let $c(s(t)=0)=0$ for all $t\in[0,t]$.
	\item $f\equiv 0$.
	\end{enumerate}
	 Let $\hat u$ is a strongly robust optimal control for the performance functional
	\begin{equation*}
	u\rightarrow J^{\P}(u)
	\end{equation*}
	for every probability measure $\P\in\mathcal{P}$. Then
	 \begin{equation}\label{eq_necessary_max_hamilton_G2}
		\frac{\partial \hat H^{G}}{\partial u}(t):=\frac{\partial}{\partial u} H(t,\hat X(t), u, \hat p^{G}(t),\hat q^{G}(t))=0,\ q.s.
	\end{equation}
\end{theorem}

\begin{proof}
	Fix a probability measure $\P\in\mathcal{P}$. By Lemma \ref{lemma_necessMax} we know that $\hat u$ is a critical point ($\P$-a.s.) of the  Hamiltonian
	\[
		\frac{\partial}{\partial u} H(t,\hat X(t), \hat u,\hat p^{\P}(t),\hat q^{\P}(t))=0,\ \forall\ t\in[0,T].
	\]
	Using this fact we get
	\begin{align*}
		0&=\frac{\partial }{\partial u} H(t, \hat X(t), \hat u, \hat p^{\P}(t), \hat q^{\P}(t))\\
			&=\left[\hat X(t)m(t)\hat p^{\P}(t)+\hat X(t)s(t)\hat q^{\P}(t) \right]\frac{d\langle B\rangle (t)}{dt}.
	\end{align*}
	By the assumption on the process $s$ we compute that
	\[
		\hat q^{\P}(t)=-\frac{m(t)}{s(t)}\hat p^{\P}(t).
	\]
	But then we see that $\hat p^{\P}$ has dynamics
	\begin{align*}
		d\hat p^{\P}(t)&=-\frac{\partial}{\partial x}H(t,\hat X(t),\hat u(t), \hat p^{\P}(t), -\frac{m(t)}{s(t)}\hat p^{\P}(t))dt-\frac{m(t)}{s(t)}\hat p^{\P}(t)dB(t)\\
		&=-\frac{m(t)}{s(t)}\hat p^{\P}(t)dB(t).
	\end{align*}
	Hence 
	\[
		\hat p^{\P}(t)=\E^{\P}[g'(\hat X(T))|\F_t]\quad \P-a.s.
	\]
	We also remember that
	\[
		\hat p^{G}(t)=\sideset{}{^{\P}}\esup_{\Q\in\mathcal{P}(t,\P)}\hat p^{\Q}(t)=\sideset{}{^{\P}}\esup_{\Q\in\mathcal{P}(t,\P)}\E^{\Q}[g'(\hat X(T))|\F_t]\quad \P-a.s.
	\]
	Thus by the characterization of the conditional $G$-expectation in \eqref{eq_cond_exp_rep} we obtain that $\hat p^{G}(t)$ is a $G$-martingale with representation
	\[
		\hat p^G(t)=\hE[g'(\hat X(T))|\F_t]=\hE[g'(\hat X(T))] + \int_0^t\hat q^G(s)dB(s)+\hat K(t)\quad q.s.
	\]
	and consequently it has dynamics
	\[
		d\hat p^{G}(t)=\hat q^G(t)dB(t)+d\hat K(t).
	\]
	But in that case we know that for almost all $t\in[0,T]$ we must have that
	\[
		0=\frac{\partial}{\partial x}H(t,\hat X(t),\hat u(t),\hat p^G(t),\hat q^G(t))=\hat u(t)[m(t)\hat p^G(t)+s(t) \hat q^G(t)]\frac{d\langle B\rangle (t)}{dt}\ q.s.
	\]
	By assumption on $\hat u$ we conclude that 
	\[
		m(t)\hat p^G(t)+s(t) \hat q^G(t)=0\ q.s.
	\]
	Hence
	\[
		 \hat q^G(t)=-\frac{m(t)}{s(t)}\hat p^{G}(t)
	\]
	and we can easily check then that
	\[
		\frac{\partial}{\partial u}H(t,\hat u,\hat p^G(t),\hat q^G(t))=0.
	\]
\end{proof}

\section{Examples}\label{section:ex}

We now consider some examples to illustrate the previous results. In the sequel we assume to work with a one-dimensional  $G$-Brownian motion with operator $G$ of the form
\begin{equation}\label{G-operator}
	G(a):=\frac{1}{2}(a^+ - \underline {\sigma}^2a^-),\quad \underline {\sigma}^2>0,
\end{equation}
i.e. with quadratic variation $\langle B\rangle (t)$ lying within the bounds $\underline {\sigma}^2t$ and $t$.

\subsection{Example I}
Consider 
\begin{equation}
dX(t)= dB(t)-c(t) dt.
\end{equation}
where $c(t)$, $t\in [0,T]$, is stochastic process such that $c(t)\in L^1_G(\Omega_t)$ for all $t\in [0,T]$. We wish to solve the optimal control problem for every  $\P\in\mathcal{P}$ under the performance criterion

\begin{equation}\label{pr_ex1}
J^{\P}(c)=\E^{\P}\left[\int_0^T \ln c(t) dt +X(T)\right].
\end{equation}
In the notation of Section \ref{section:suff_maxprinciple}, we have chosen here $f(t,x,c)=\ln c$ and $g(x)=x$, i.e. $g'(x)=1$.
Then the Hamiltonian is given by
\begin{equation}\label{eq_hamEx1}
H(t,x,c,p,q)=\ln c + q \frac{d\langle B\rangle_t }{dt}- c p, 
\end{equation}
and by \eqref{eq_p} we obtain
\begin{align}
dp(t)&= q(t) dB(t);\ 0\leq t\leq T, \label{eq_pqex1}\\ 
p(T)&=g'(X(T))=1,\nonumber
\end{align}
i.e. $q=0, p=1$. Furthermore by \eqref{eq_hamEx1} we have
\begin{equation*}
\frac{\partial H}{\partial c}=\frac{\partial }{\partial c}[\ln c -cp]=\frac{1}{c} -p,
\end{equation*}
i.e. $\hat c(t)=1$, $t\in [0,T]$, is strongly robust optimal by Theorem \ref{theo_suff}. 

Note that by the proof we could choose a general utility function instead of logarithmic utility without losing the existence of the strongly robust optimal control.



\subsection{Example II}

Consider
\begin{equation}
dX(t)= X(t)[b(t)dt + dB(t)]-c(t) dt,
\end{equation}
and Problem \eqref{pr_ex1}.  Here $b(t)$ is a deterministic measurable function. Then the Hamiltonian is given by
\begin{equation}\label{eq_hamEx1_2}
H(t,x,c,p,q)=\ln c + xq \frac{d\langle B\rangle_t }{dt}+ (xb(t)- c)) p.
\end{equation}
Here 
\begin{align}
dp(t)&= -\left(b(t)p(t) + q(t)\frac{d\langle B\rangle_t }{dt}\right) dt + q(t) dB(t);\ 0\leq t\leq T,\\ \label{eq_pqex1_2}
p(T)&=g'(X(T))=1.\nonumber
\end{align}
Put $q=0$, then
\begin{align*}
dp(t)&= -b(t)p(t)  dt,\\
p(T)&=1,
\end{align*}
i.e. $p(t)=\exp{\int_t^T b(s) ds}$ and $\hat c(t)=\frac{1 }{p(t)}$  is strongly robust optimal by Theorem \ref{theo_suff}. 

\subsection{Example III}
Consider the Merton-type problem with the logarithmic utility: let 
\[
	dX^u(t)=X^u(t)\left[m(t)u(t)d\langle B\rangle (t)+s(t)u(t)dB(t)\right]
\]
where $u(t)\in L^2_G(\Omega_t)$ for all $t\in [0,T]$ and $m$ and $s$ are two deterministic functions. Assume that $s(t)\neq 0$ for all $t\in[0.T]$. We are interested in to find a strongly robust optimal control problem for the family of probability measures $\mathcal{P}$ with the performance criterion given by
\[
	J^{\P}(u):=\E^{\P}[\ln X^u(T)].
\]
The Hamiltonian associated with this problem is given by
\begin{equation}\label{eq_hamiltotnian_3}
	H(t,x,u,p,q)=xu[m(t)p+s(t)q]\frac{d\langle B\rangle}{dt}(t)
\end{equation}
and for each admissible control $u$ we consider adjoint $G$-BSDE of the form
\begin{align*}
	dp(t)&=-u(t)[m(t)p(t)+s(t)q(t)]d\langle B\rangle(t)+g(t)dB(t)+dK(t)\\
	p(T)&=X^{-1}(T).
\end{align*}
Note that the adjoint equation is linear, hence  by Remark 3.3 in \cite{Girsanov} we obtain the representation formula for the solution 
\[
	p(t)=X^{-1}(t)\hE[X(T)X^{-1}(T)|\F_t]=X^{-1}(t).
\]
Moreover, by the dynamics of $X^{-1}$ we deduce that
\[
	q(t)=-u(t)s(t)p(t),\quad K\equiv 0.
\]
Plugging this solution into the Hamiltonian \eqref{eq_hamiltotnian_3} we get that
\[
	H(t,X^u(t),v,p(t),q(t))=X^u(t)v[m(t)-u(t)s^2(t)]p(t)\frac{\langle B\rangle}{dt}(t),
\]
hence the critical point of the Hamiltonian must satisfy
\[
	\hat u(t)=\frac{m(t)}{s^2(t)}
\]	
and this is our strongly robust optimal control.

Note that we can also solve this problem directly by omega-wise maximization, without using the maximum principle and $G$-BSDE's. 
In fact we may consider more general dynamics in $X$
\[
	dX^u(t)=X^u(t)\left[b(t)u(t)dt+m(t)u(t)d\langle B\rangle (t)+s(t)u(t)dB(t)\right]
\]
and by direct computation it might be checked that the strongly robust optimal control takes the form
\[
	\hat u(t)=\frac{b(t)+m(t)\frac{d\langle B\rangle}{dt}(t)}{s^2(t)\frac{d\langle B\rangle}{dt}(t)}.
\]	
However it is important to note that this control is not quasi-continuous any more (see \cite{Song_unique}) and it doesn't have sense to consider $G$-BSDE's associated with such a control.

\section{Counterexample: the Merton problem with the power utility}

In this example we consider the Merton problem with the power utility and show that generally we cannot drop the assumption 
$\hat K\equiv 0$ without losing the strong sense of the optimality. First, we solve the classical robust utility maximization problem and then we prove that the optimal control for that problem is optimal usually only in a weaker sense, i.e. there exists a probability measure $\mathbb P\in \mathcal P$ such that the  control is not optimal under $\mathbb P$, even though the control satisfies all the conditions of the sufficient maximum principle with the exception of $\hat K\equiv 0$.

Consider first the classical robust utility maximization problem
\[
	u\mapsto \hat{J}(u):=\hE[\int_0^Tf(t,X(t),u(t))dt+g(X(T))],
\]
where $X$ has dynamics for any $u\in\A$
\[
	dX(t)=m(t)X(t)u(t)d\langle B\rangle (t)+s(t)X(t)u(t)dB(t).
\]
Then 
\[
	X(t)=x\exp\{\int_0^ts(r)u(r)dB(r)+\int_0^t[m(r)u(r)-\frac{1}{2}s^2(r)u^2(r)]d\langle B\rangle(r)\}.
\]
We assume that $m$ and $s$ are bounded and deterministic and $s\neq0$. Put $f\equiv 0$ and $g(x)=\frac{1}{\alpha}x^{\alpha},\ \alpha \in ]0,1[$. Hence 
\begin{align*}
	\hat J(u)&=\frac{x}{\alpha}\hE[\exp\{\alpha\int_0^Ts(r)u(r)dB(r)+\alpha\int_0^T[m(r)u(r)-\frac{1}{2}s^2(r)u^2(r)]d\langle B\rangle(r)\}]\\
	&=\frac{x}{\alpha}\hE[\exp\{\alpha\int_0^Ts(r)u(r)dB(r)-\frac{\alpha^2}{2}\int_0^Ts^2(r)u^2(r)d\langle B\rangle(r)\}\cdot\\
	&\quad \cdot\exp\{\int_0^T[\alpha m(r)u(r)+\frac{\alpha^2-\alpha}{2}s^2(r)u^2(r)]d\langle B\rangle(r)\}].
\end{align*}
We now use the Girsanov theorem for $G$-expectation and the $G$-martingale 
$$
M(t):=\exp\{\alpha\int_0^ts(r)u(r)dB(r)-\frac{\alpha^2}{2}\int_0^ts^2(r)u^2(r)d\langle B\rangle(r)\},
$$ 
see Section 5.2. in \cite{Girsanov}. We get the sublinear expectation $\hE^{u}$ under which the process ${B}^u(t):=B(t)-\int_0^ts(r)u(r)d\langle B\rangle (r)$ is a $G$-Brownian motion. Note that 
\begin{equation}\label{eq_GB}
\langle B^u \rangle(t)=\langle
 B\rangle(t)
 \end{equation}
q.s. Moreover it is easy to check that the deterministic control 
 \[\hat u(r)=\frac{m(r)}{(1-\alpha)s^2(r)}\] 
 is a maximizer of the following function
 \[
	u\mapsto \alpha m(r)u +\frac{\alpha^2-\alpha}{2}s^2(r)u^2.
 \]
 Hence we get that
 \begin{align}
	\hat J(u)
	&=\frac{x}{\alpha}\hE^{u}[\exp\{\int_0^T[\alpha m(r)u(r)+\frac{\alpha^2-\alpha}{2}s^2(r)u^2(r)]d\langle B\rangle(r)\}]\notag\\
	&\leq \frac{x}{\alpha}\hE^{u}[\exp\{\int_0^T[\alpha m(r)\hat u(r)+\frac{\alpha^2-\alpha}{2}s^2(r)(\hat u)^2(r)]d\langle  B^u\rangle(r)\}]\notag\\
	&= \frac{x}{\alpha}\hE^{\hat u}[\exp\{\int_0^T[\alpha m(r)\hat u(r)+\frac{\alpha^2-\alpha}{2}s^2(r)(\hat u)^2(r)]d\langle  B^{\hat u}\rangle(r)\}]=\hat J(\hat u).\label{eq_ex3_Girsanov}
\end{align}
The last equalities are consequence of \eqref{eq_GB} and of the fact that the integrand is deterministic and that $B^u$ and $B^{\hat u}$ are $G$-Brownian motions under $\E^u$ and $\E^{\hat u}$ (respectively). Equation \eqref{eq_ex3_Girsanov} shows then that $\hat u$ is an optimal control for this weaker optimization problem. 

Now consider the adjoint equation related to $\hat u$ in terms of a $G$-BSDE. The backward equation is linear due to linearity of the Hamiltonian, hence we may use the conditional expectation representation of a linear $G$-BSDE's (compare with Remark 3.3 in \cite{Girsanov}):
\begin{align*}
	\hat p^G(t)&=\frac{1}{\hat X(t)}\hE\left[(\hat X(T))^{\alpha-1}\hat X(T)|\mathcal{F}_t\right]\\
	&=(\hat X(t))^{\alpha-1} \hE\left[\exp\{\alpha\int_t^Ts(r)\hat u(r)dB(r)-\frac{\alpha^2}{2}\int_t^Ts^2(r)\hat u^2(r)d\langle B\rangle(r)\}\cdot\right.\\
	&\quad \left.\cdot\exp\{\int_t^T[\alpha m(r)\hat u(r)+\frac{\alpha^2-\alpha}{2}s^2(r)\hat u^2(r)]d\langle B\rangle(r)\} \LARGE| \mathcal{F}_t\right].
\end{align*}
Applying the Girsanov theorem and the same reasoning as in \eqref{eq_ex3_Girsanov} we easily get that
\begin{align*}
	\hat p^G(t)&=\frac{1}{\hat X(t)}\hE\left[(\hat X(T))^{\alpha-1}\hat X(T)|\mathcal{F}_t\right]\\
	&=(\hat X(t))^{\alpha-1} \hE^{\hat u}\left[\exp\{\int_t^T[\alpha m(r)\hat u(r)+\frac{\alpha^2-\alpha}{2}s^2(r)\hat u^2(r)]d\langle B\rangle(r)\} \LARGE| \mathcal{F}_t\right]\\
	&=(\hat X(t))^{\alpha-1} \hE^{\hat u}\left[\exp\{\int_t^T\frac{\alpha}{2(1-\alpha)}\frac{m^2(r)}{s^2(r)}d\langle B\rangle(r)\} \LARGE| \mathcal{F}_t\right]\\
	&=(\hat X(t))^{\alpha-1} \hE^{\hat u}\left[\exp\{\int_t^T\frac{\alpha}{2(1-\alpha)}\frac{m^2(r)}{s^2(r)}d\langle B^{\hat u}\rangle(r)\} \LARGE| \mathcal{F}_t\right]\\
	&=(\hat X(t))^{\alpha-1} \hE\left[\exp\{\int_t^T\frac{\alpha}{2(1-\alpha)}\frac{m^2(r)}{s^2(r)}d\langle B\rangle(r)\} \LARGE| \mathcal{F}_t\right].
\end{align*}
Furthermore we also know that the integrand is always positive by the assumption $\alpha \in ]0,1[$, hence we get by the representation of the conditional $G$-expectation \eqref{eq_cond_exp_rep} that for every $\P\in\mathcal P$ and 
by \eqref{G-operator} that
\begin{align*}
\hE&\left[\exp\{\int_t^T\frac{\alpha}{2(1-\alpha)}\frac{m^2(r)}{s^2(r)}d\langle B\rangle(r)\} \LARGE| \mathcal{F}_t\right]\\&=\sideset{}{^{\P}}\esup_{\P'\in\mathcal{P}(t,\P)}\E^{\P'}\left[\exp\{\int_t^T\frac{\alpha}{2(1-\alpha)}\frac{m^2(r)}{s^2(r)}d\langle B\rangle(r)\} \LARGE| \mathcal{F}_t\right]\\
&=\exp\{\int_t^T\frac{\alpha}{2(1-\alpha)}\frac{m^2(r)}{s^2(r)}dr\}\quad \P-a.s.
\end{align*}
Hence
\begin{align*}
	\hat p^G(t)
	&=(\hat X(t))^{\alpha-1} \exp\{\int_t^T\frac{\alpha}{2(1-\alpha)}\frac{m^2(r)}{s^2(r)}dr\}=:(\hat X(t))^{\alpha-1}\cdot Z(t).
\end{align*}
By integration by parts for $\hat X^{-1}$ and $Z$ one can compute that
\begin{align}
	d\hat p^G(t)&=-\frac{m(t)}{s(t)}\hat p^G(t)dB(t)+\frac{\alpha m^2(t)}{2(1-\alpha)s^2(t)}\hat p^G(t)(d\langle B\rangle (t)-dt).\label{eq_ex3_dynamics}
\end{align}
By comparing equation \eqref{eq_ex3_dynamics} with the adjoint equation \eqref{eq_p} we obtain first that
\[
	\hat q^G(t)=-\frac{m(t)}{s(t)}\hat p^G(t)
\]
and hence that $\hat u$ is a maximizer of the function $u\mapsto H(t,\hat X(t),u,\hat p^G(t), \hat q^G(t))$. Secondly, we get that the process $\hat K$  has the explicit form
\[
	\hat K(t)=\int_0^t \frac{\alpha m^2(r)}{2(1-\alpha)s^2(r)}\hat p^G(r)(d\langle B\rangle (r)-dr)
\]
and, consequently  is a non-trivial process. 

To summarize the example so far: we have shown that $\hat u$ is optimal in a weaker sense. We also showed that it satisfies the assumption for the necessary maximum principle for strongly robust optimality and that all assumptions of the sufficient maximum principle are satisfied, with the exception of the vanishing of the process $\hat K$. Now we prove that $\hat u$ is not optimal in the stronger sense, hence the assumption on the process $\hat K$ is really crucial for our result and cannot be dropped.

Fix $\P \in \mathcal P$ and assume that $\hat u$ is optimal under $\P$. By Lemma \ref{lemma_necessMax} we know that $\hat u$ is a critical point of the Hamiltonian evaluated in $\hat p^{\P}$ and $\hat q^{\P}$. Hence, by the same analysis as in Theorem \ref{theorem_necessMax_Merton} we see that
\[
	d\hat p^{\P}(t)=-\frac{m(t)}{s(t)}\hat p^{\P}(t)dB(t),
\]
therefore
\begin{equation}\label{eq_ex3_pT1}
	\hat p^{\P}(T)=\hat p^{\P}(0)\exp \left\{-\int_0^T\frac{m(t)}{s(t)}dB(t)-\frac{1}{2}\int_0^T\frac{m^2(t)}{s^2(t)}d\langle B\rangle(t)\right\}.
\end{equation}

However, we know by the dynamics of $\hat X$ and the terminal condition of $\P$-BSDE that
\begin{align}
	\hat p^{\P}(T)&=(\hat X(T))^{\alpha-1}\notag\\
	&=x^{\alpha-1}\exp\left\{ (\alpha-1)\left[\int_0^T\hat u(t)s(t)dB(t)+\int_0^T \left(\hat u(t)m(t)-\frac{1}{2}\hat u^2(t)s^2(t)\right)d\langle B\rangle (t) \right]\right\}\notag
	\\&=x^{\alpha-1}\exp\left\{ -\int_0^T\frac{m(t)}{s(t)}dB(t)-\frac{1}{2}\int_0^T\frac{m^2(t)(1-2\alpha)}{s^2(t)(1-\alpha)}d\langle B\rangle(t)\right\}.\label{eq_ex3_pT2}
\end{align}

Dividing \eqref{eq_ex3_pT1} by \eqref{eq_ex3_pT2} we get that
\[
	1=\frac{\hat p^{\P}(0)}{x^{\alpha-1}}\exp\left\{\int_0^T\frac{\alpha m^2(t)}{2s^2(t)(\alpha-1)}d\langle B\rangle(t)\right\}.
\]
The equalities here are $\P$-a.s. so we get that the integral $\int_0^T\frac{\alpha m^2(t)}{2s^2(t)(\alpha-1)}d\langle B\rangle(t)$ must be equal $\P$-a.s. to a constant. However the quadratic variation of the canonical process under $\P$ is generally a non-deterministic stochastic process, hence also the integral is a random variable, in general non-constant. This shows that $\hat u$ is optimal under $\P$ only for very specific probability measures such as the Wiener measure.

To conclude, $\hat u$ is not optimal for every probability measure $\P\in\mathcal{P}$ even though it is a maximizer of the Hamiltonian related to $\hat u$. This example shows that the new strong notion of optimality is rather restricted and we may expect it only in very special cases when the process $\hat K$ vanishes. 

\bibliographystyle{plain}
\bibliography{biblio_max}

\begin{thebibliography}{10}

\bibitem{Martini}
Denis L. and Martini C.
\newblock A theoretical framework for the pricing of contingent claims in the
  presence of model uncertainty.
\newblock {\em The Annals of Applied Probability}, 16:827--852, 2006.

\bibitem{Denis_function_spaces}
Denis L., Hu~M., and Peng. S.
\newblock Function spaces and capacity related to a sublinear expectation:
  application to {$G$}-{B}rownian motion paths.
\newblock {\em Potential Analysis}, 34:139--161, 2011.

\bibitem{Peng_bsde}
Hu~M., Peng~S. Ji~S., and Song Y.
\newblock Backward stochastic differential equations driven by {$G$}-{B}rownian
  motion.
\newblock {\em Stochastic Processes and their Applications}, 124:759--784,
  2014.

\bibitem{Girsanov}
Hu~M., Ji~S., and Peng S.
\newblock Comparison theorem, {F}eynman-{K}ac formula and {G}irsanov
  transformation for {BSDE}s driven by {$G$}-{B}rownian motion.
\newblock {\em Stochastic Processes and their Applications}, 124:1170--1195,
  2014.

\bibitem{utility2}
Hu~M., Ji~S., and Yang S.
\newblock A stochastic recursive optimal control problem under the
  {$G$}-expectation framework.
\newblock Preprint, arXiv:1306.1312, 2013.

\bibitem{Soner_mart_rep}
Soner M., Touzi N., and Zhang J.
\newblock Martingale representation theorem for the {$G$}-expectation.
\newblock {\em Stochastic Processes and their Applications}, 121:265--287,
  2011.

\bibitem{Soner_quasi_anal}
Soner M., Touzi N., and Zhang J.
\newblock Quasi-sure stochastic analysis through aggregation.
\newblock {\em Electronic Journal of Probability}, 16:1844--1879, 2011.

\bibitem{Soner_2bsde}
Soner M., Touzi N., and Zhang J.
\newblock Wellposedness of second order backward {SDEs}.
\newblock {\em Probability Theory and Related Fields}, 153:149--190, 2011.

\bibitem{utility1}
A.~Matoussi, Possamai D., and Zhou C.
\newblock Robust utility maximization in non-dominated models with {2BSDEs}.
\newblock {\em Mathematical Finance}, 2013.
\newblock DOI: 10.1111/mafi.12031.

\bibitem{Complete_mart}
Zhang~J. Peng~S., Song~Y.
\newblock A complete representation theorem for {$G$}-martingales.
\newblock Preprint, arXiv:1201.2629v1, 2012.

\bibitem{Peng_GBM}
Peng S.
\newblock {$G$}-expectation, {$G$}-{B}rownian motion and related stochastic
  calculus of {I}t\^o type.
\newblock {\em Stochastic Analysis and Applications}, 2:541--567, 2007.

\bibitem{Peng_skrypt}
Peng S.
\newblock Nonlinear expectations and stochastic calculus under uncertainty.
\newblock Preprint, arXiv1002.4546v1, 2010.

\bibitem{Song_Mart_decomp}
Song Y.
\newblock Some properties on {$G$}-evaluation and its applications to
  {$G$}-martingale decomposition.
\newblock {\em Science China}, 54:287--300, 2011.

\bibitem{Song_unique}
Song Y.
\newblock Uniqueness of the representation for {$G$}-martingales with finite
  variation.
\newblock {\em Electronic Journal of Probability}, 17:1--15, 2012.

\end{thebibliography}

\end{document}